\documentclass[12pt]{amsart}

\usepackage{amsfonts,amsmath,latexsym,amssymb,verbatim,amsbsy,amsthm}

\usepackage[top=1in, bottom=1in, left=1in, right=1in]{geometry}

\usepackage[dvipsnames]{xcolor}

\usepackage[colorlinks=true, pdfstartview=FitV, linkcolor=RoyalBlue,citecolor=ForestGreen, urlcolor=blue]{hyperref}

\theoremstyle{plain}
\newtheorem{THEOREM}{Theorem}[section]

\newtheorem{theorem}[THEOREM]{Theorem}
\newtheorem{corollary}[THEOREM]{Corollary}
\newtheorem{lemma}[THEOREM]{Lemma}

\theoremstyle{definition}

\theoremstyle{remark}

\newtheorem{remark}[THEOREM]{Remark}

\newcommand{\thm}[1]{Theorem~\ref{#1}}
\newcommand{\lem}[1]{Lemma~\ref{#1}}

\def \a {\alpha}
\def \b {\beta}
\def \d {\delta}
\def \D {\Delta}
\def \g {\gamma}

\def \e {\varepsilon}
\def \f {\varphi}

\def \l {\lambda}
\def \L {\Lambda}
\def \n {\nabla}
\def \r {\rho}

\def \O {\Omega}

\def \bs {\boldsymbol{\sigma}}

\def \cC {\mathcal{C}}

\def \cS {\mathcal{S}}
\def \cF {\mathcal{F}}

\def \dx  {\, \mbox{d}x}

\def \dy  {\, \mbox{d}y}

\def \dxi  {\, \mbox{d}\xi}

\def \Hdiv {H^m_{\mathrm{div}}}

\newcommand{\N}{\ensuremath{\mathbb{N}}}   %%% naturals
\newcommand{\Z}{\ensuremath{\mathbb{Z}}}   %%% integers
\newcommand{\R}{\ensuremath{\mathbb{R}}}   %%% reals
\newcommand{\C}{\ensuremath{\mathbb{C}}}   %%% complex
\newcommand{\T}{\ensuremath{\mathbb{T}}}   %%% torus

\renewcommand{\S}{\ensuremath{\mathbb{S}}} %%% sphere
\def \lan {\langle}
\def \ran {\rangle}
\def \p {\partial}
\def \ra {\rightarrow}
\def \ss {\subset}
\def \bs {\backslash}

\newcommand{\rest}[2]{#1\raisebox{-0.3ex}{\mbox{$\mid_{#2}$}}}

\DeclareMathOperator{\supp}{supp} %
\DeclareMathOperator{\diver}{div} %
\DeclareMathOperator{\im}{Im} %
\DeclareMathOperator{\op}{Op} %
\DeclareMathOperator{\Id}{Id} %

\def \Ldiv {L^2_{\mathrm{div}}(\O)}
\def \comp {\mathrm{comp}}
\def \ess {\mathrm{ess}}

\begin{document}

\title{On the Rayleigh-Taylor instability in presence of a background shear}

\author{Roman Shvydkoy}
\address{Department of Mathematics, Statistics, and Computer Science, M/C 249,\\
    University of Illinois, Chicago, IL 60607, USA}
\email{shvydkoy@uic.edu}

\begin{abstract}In this note we revisit the classical subject of the Rayleigh-Taylor instability in presence of an incompressible background shear flow. We derive a formula for the essential spectral radius of the evolution group generated by the linearization near the steady state and reveal that the velocity variations neutralize shortwave instabilities. The formula is a direct generalization of the result of H. J. Hwang and Y. Guo in the hydrostatic case \cite{HG}. Furthermore, we construct a class of steady states which posses unstable discrete spectrum with neutral essential spectrum. The technique involves the WKB analysis of the evolution equation and contains novel compactness criterion for pseudo-differential operators on unbounded domains. 
    \end{abstract}

\thanks{The author thanks Zhiwu Lin and Chongchung Zeng for many fruitful conversations and
	 Georgia Institute of Technology for hospitality. This research was partially supported by NSF grant DMS 1515705 and the College of LAS, UIC}



\maketitle

\section{Introduction}  This note revisits the classical subject of Rayleigh-Taylor instability -- when an inhomogeneous fluid is subjected to the gravitation force $\vec{g}$  and if heavier fluid occurs on top of lighter fluid it naturally tries to overturn, \cite{Ch1961,Rayleigh}. The dynamics is described by a system of Euler equation (we assume the fluid is ideal) given by
\begin{equation}\label{e:RT}
\begin{split}
\r(u_t + u \cdot \n u) + \n p &= \r \vec{g} \\
\r_t + u \cdot \n \r &= 0\\
\n \cdot u & = 0.
\end{split}
\end{equation}
Here $\vec{g} = \lan 0,-g \ran$, and $g$ is the gravitational acceleration. We assume the fluid is confined to the strip $\O = \T \times \R$, and is periodic in the first coordinate $x_1$. 

A stratified smooth density $\rho_0(x_2)$ and zero velocity $u_0=0$ define a hydrostatic equilibrium with pressure gradient balancing out the gravitation force: $\n p_0 = \rho_0 \vec{g}$. Assuming that $\rho_0(\pm \infty) = \rho_{\pm}$, $0< \rho_{-} < \rho_+<\infty$ the fluid turns into an unstable state. Rigorous analysis of the eigenvalue problem for the linearized system around such steady state (and in fact more generally in presence of a compressible shear $u_0=\lan 0, U(x_2)\ran $) was performed by Lafitte et al \cite{Laf2001,HeLaf2003,CCLR2001,Laf2008} exploiting the variational nature of the resulting system. The work of Hwang and Guo \cite{HG} gives a complete spectral analysis of the hydrostatic case showing that the maximal exponential growth rate $\L$ of the $C_0$-semigroup is given by 
\begin{equation}\label{}
\L^2 = \sup_{v\in L^2(\O)} \frac{ \int g \r'_0 v^2 dx}{\int \r_0 v^2 dx} = \sup_{x: \r'_0(x)>0} \sqrt{   \frac{g \r_0'(x)}{\r_0(x)} } .
\end{equation}
Moreover $\L$ is a limit point of a sequence of exact eigenvalues $\l_k \to \L$, making $\L$ a point of the essential spectrum. The work \cite{HG} extends further to prove the nonlinear instability of hydrostatic states in $H^s$ for $s\geq 3$, and Lafitte extends these results in \cite{Laf2008} to include quasi-isobaric density profiles which allow for $\rho_- = 0$. 

In this present work we study linear instability for more general steady states which include background parallel shear $u_0=\lan U(x_2), 0 \ran$. The linearization around $(u_0,\r_0)$ takes form
\begin{equation}\label{e:LRT}
\begin{split}
v_t + u_0 \cdot \n v + v \cdot \n u_0 + \frac{1}{\r_0} \n q & = \frac{\vec{g}}{\r_0} r \\
r_t + u_0 \cdot \n r + v \cdot \n \r_0 &= 0\\
\n \cdot v & = 0.
\end{split}
\end{equation}
This system does not seem to retain variational formulation and thus its spectral analysis becomes quite different from the hydrostatic case. At present it is not known whether linearly unstable states are also non-linearly unstable. The difficulty in proving such a statement for conservative systems lies in the presence of continuous, or essential, spectrum. Even for the homogeneous 2D Euler equation this Lyapunov-type theorem is an outstanding open problem despite recent strong efforts by Lin and Zeng \cite{LZ,L-survey,L}, Friedlander et al  \cite{FSV,VF2003}.  In search for the eigenmodes in classical formulation 
\begin{equation}\label{e:eigen}
\begin{split}
v &= \n^\perp( \phi(y) e^{ik x}) e^{\l t} \\
r& = r(y) e^{ik x} e^{\l t},
\end{split}
\end{equation}
where $\phi$ and $r$ are complex unknown functions, and $\l \in \C$ is a sought after eigenvalue, we obtain the full Rayleigh-type system 
\begin{equation}\label{e:Rayl}
\begin{split}
k^2 \r_0 \phi - (\r_0 \phi')' + \frac{(\r_0 U')'}{U - c} \phi &= - \frac{g \r'_0 }{(U - c)^2} \phi\\
(U - c) r + \phi \r_0' & = 0,
\end{split}
\end{equation}
where $c = i\frac{\l}{k}$. The related well-studied counterpart of \eqref{e:Rayl} is obtained by performing Boussinesq approximation in which $\r_0$ is assumed to vary little compared to $\r_0'$. Thus, the density on the left hand side of \eqref{e:Rayl} is replaced with an averaged constant density $\bar{\r_0}$. The resulting system, called the Taylor-Goldstein equation, is more amenable to analysis. The unstable modes were constructed by Friedlander using the method of continued fractions, \cite{fried}. 

In this work we study the full system \eqref{e:LRT} and prove two results. Let $G_t$ denote the evolution operator at time $t$ generated by \eqref{e:LRT}. Note that $\{G_t\}_{t\in \R}$ defines a strongly continuous $C_0$-group on any Sobolev space $H^m$, $m\in \R$.  First we prove that the essential spectral radius of $G_t$ on any $H^m$ is given by the formula
\begin{equation}\label{e:ress}
r_{\ess}(G_t, H^m) = e^{t \mu},
\end{equation}
where the exponent $\mu$ is given by
\begin{equation}\label{e:mu}
\mu = \sup_{x: U'(x) = 0, \r'_0(x)>0} \sqrt{   \frac{g \r_0'(x)}{\r_0(x)} },
\end{equation}
where $\mu=0$ if the set over which the supremum is taken is empty. This formula generalizes the result of Hwang and Guo \cite{HG} to non-hydrostatic case. It also demonstrates a surprising stabilization effect of the background shear: if the shear is not uniform $U' \neq 0$ at the unstable points $\rho_0'>0$, then the essential spectrum is neutral making instability purely large scale, i.e. coming from the discrete part of the spectrum. 
 
Second, we construct a class of steady states $(U,\rho_0)$ for which $\mu=0$, yet the Rayleigh system \eqref{e:Rayl} has a non-trivial solution with unstable spectrum, \thm{t:Rayl}.  The argument goes by perturbation from an unstable hydrostatic state $(0,\rho_0)$ along an arbitrary shear $U$ with $U' \neq 0$. This provides a wide range of examples of steady states with neutral essential spectrum yet non-trivial unstable point spectrum. The constructed states are posed to be non-linearly unstable or even have local invariant manifolds, since the  obstacles coming from presence of shortwave instabilities are removed, see  Lin and Zeng \cite{LZ} for the case of the homogeneous Euler equations. We will leave these questions for future research.  
 
At the core of our analysis is the geometric optics approach that has been successful in describing shortwave instabilities in a variety of fluid models, most recently see \cite{S-cont,SL,S-ess} and literature therein. Formula \eqref{e:ress} is an analogue of Vishik's result \cite{V} for the spectral radius on the incompressible Euler system and of author's general result \cite{S-ess} for advective systems on periodic domains.  The major difficulty presented by this particular situation consists of lack of compactness of the underlying fluid domain $\R \times \T$, which necessitates many of the extra technical argumentations to be made in description of the microlocal structure of the group $G_t$.  Following the strategy developed in prior works we seek to show that the group operator $G_t$ is given, up to a compact perturbation, by a pseudo-differential operator $\op[B_t]$ with symbol $B_t(x,\xi)$ generated by a bicharacteristic-amplitude system obtained from \eqref{e:LRT} by reading off the leading order term in the WKB ansatz 
\[
V(x,t) = b(x,t) e^{i S(x,t)/\e} + O(\e), \quad V = (v,r).
\]
To justify this statement we will devote first half of this article to extend the classical pseudo-differential calculus into the settings of mixed domains $\O =\T^d \times \R^{n-d}$ with specific purpose to derive new compactness criteria for PDOs on such domains suitable for our settings. The main result  formulated  in \thm{t:compact}  may be of independent interest.  Our criterion states that if the symbol $a(x,y,\xi)$  is of negative order in frequency $\xi$ (necessary condition) and  decays at spacial infinity $x,y \to \infty$ on any near-diagonal region $|x-y|<R$, then $\op[a]$ is compact on $\O$, see classical texts \cite{shubin,treves} for discussion on local compactness.

\section{Pseudo-differential calculus on mixed domains}

In this section we review some facts about pseudo-differential operators (PDO for short) on mixed periodic-open space domains $\O =\T^d \times \R^{n-d}$, where $1\leq d <n$, where $\T^d$ denotes the torus with $2\pi$ periods. Let  $dy$ denote the Haar measure on $\O$ (product of the usual Lebesgue on $\R^{n-d}$ and normalized Lebesgue on $\T^d$), and let $d\xi$ be the corresponding Haar measure on the dual group $\O^* = \Z^d \times \R^{n-d}$ (i.e. Lebesgue on $\R^{n-d}$ and counting on $\Z^d$). Our first goal is to make sense of the expression
\begin{equation}\label{mixedPDO}
Au(x) = \int_{\O^*}\int_{\O} a(x,y,\xi)e^{i(x-y)\cdot \xi} u(y) \dy \dxi.
\end{equation}
In the open-space dimensions we can apply the classical local theory. The PDOs on purely periodic domains can still be defined classically as on a compact manifold. However, it is more desirable to use of the explicitly global structure of the operator as defined by \eqref{mixedPDO}. Symbolic calculus of PDOs on the torus along has been developed previously in \cite{RT2010,RT2007}. As will be seen later, such results are insufficient for establishing effective boundedness and compactness criteria suitable to our application. We therefore will spend some effort to revisit the basic analysis of the operators given by \eqref{mixedPDO}.

\subsection{Class of amplitudes and the kernel of $A$}
An amplitude $a = a(x,y,\xi) \in C^\infty(\O \times \O \times \R^n)$ is said to belong to class $S^{m}_{1,0,0}(\O)$, or $S^m(\O)$ for short, if 
\[
\forall |\a|,|\b|,|\g|\leq n+1, \ \sup_{x,y,\xi} (1+|\xi|)^{|\g|-m}|\p_x^\a \p_y^\b \p_\xi^\g a(x,y,\xi)| = C_{\a,\b,\g}(a)<\infty.
\]
Note that we require $a$ to be defined for all $\xi \in \R^n$ and not just on $\O^*$.
To define a PDO \eqref{mixedPDO} we first make sense of the oscillatory sum-integral
\begin{equation}\label{}
K(x,y) = \int_{\O^*} a(x,y,\xi) e^{i (x-y)\cdot \xi} \dxi
\end{equation}
classically as a distribution on the Schwartz class $\cS(\O \times \O)$ given by
\begin{equation}\label{}
K(x,y) = \int_{\O^*}  (1+|\xi|^2)^{-N} a(x,y,\xi) (1-\D_x)^N e^{i (x-y)\cdot \xi} \dxi.
\end{equation}
for any $N>n/2$. More explicitly, for a test function $\phi \in \cS(\O \times \O)$, we have
\[
\lan K, \phi \ran =  \int_{\O^*\times \O \times \O}  (1+|\xi|^2)^{-N} \sum_{|\a + \b| \leq 2N}  c_{\a,\b} e^{i (x-y)\cdot \xi} \p_x^\b a(x,y,\xi) \p_x^\a \phi(x,y) \dxi \dx \dy,
\]
with a suitable choice of constants $c_{\a,\b}$. Thus,
\begin{equation}\label{}
K(x,y) = \int_{\O^*}  (1+|\xi|^2)^{-N}  \sum_{|\a + \b| \leq 2N} \p_x^\a[ c_{\a,\b}  \p_x^\b a(x,y,\xi)  e^{i (x-y)\cdot \xi} ] \dxi.
\end{equation}

We can then define the PDO by
\begin{equation}\label{IO}
Au(x) = \int_{\O} K(x,y) u(y) dy 
\end{equation}
interpreting integral distributionally. Amplitudes that we encounter will be of type $S^m(\O)$ for $m\leq 0$. In this case one can obtain detailed pointwise estimates on the kernel $K$ off the diagonal $x=y$. It is known that for $m=0$, $K$ is of Calderon-Zygmund type. For $m<0$, with a view towards developing compactness criteria we will need to know exactly how those bounds depend on the amplitude. The usual way of approaching this is to integrate by parts with respect to $\xi$. As this operation is prohibited in the discrete dimensions of $\O^*$, we first relate $K$ to the classical kernel over $\R^n$:
\[
K_{\R^n}(x,y) = \int_{\R^n} a(x,y,\xi) e^{i(x-y)\cdot \xi} d\xi.
\]
Let  $F_N$ denote the $d$-dimensional Fejer kernel, so that $F_N \ra \sum_{q \in \Z^d} \delta_{q}$.  We have
\[
K(x,y) = \lim_{N\to \infty}  \int_{\R^n}  a(x,y,\xi) F_N(\xi_1,\ldots,\xi_d) e^{i(x-y)\cdot \xi} \dxi.
\]
Unraveling the formula for $F_N$ we obtain
\[
K(x,y)  = \lim_{N \ra \infty} \frac{1}{N} \sum_{M = 1}^N \sum_{q\in \Z^d, |q_i| \leq M}  \int_{\R^n}  a(x,y,\xi)  e^{i(x-y + 2\pi q )\cdot \xi} \dxi.
\]
We readily obtain two representations: 
\begin{equation}\label{Kper}
\begin{split}
K(x,y) & = \lim_{N \ra \infty} \frac{1}{N} \sum_{M = 1}^N \sum_{q\in \Z^d, |q_i| \leq M} K_{\R^n}(x,y_1+2\pi q_1,\ldots,y_d+2\pi q_d, y_{d+1},\ldots,y_n) \\
&= \lim_{N \ra \infty} \frac{1}{N} \sum_{M = 1}^N \sum_{q\in \Z^d, |q_i| \leq M} K_{\R^n}(x_1+2\pi q_1,\ldots,x_d+2\pi q_d, x_{d+1},\ldots,x_n,y).
\end{split}
\end{equation}
Thus, $K$ is the Ces\`aro periodization of the open space kernel $K_{\R^n}$ only in periodic dimensions.
\begin{lemma} \label{Kbound} Suppose $a\in S^{-m}(\O)$, with $m \geq 0$. Then $K \in C^\infty(\{x\neq y\})$ and the following pointwise bounds hold
\begin{align}
| K(x,y) | & \leq \frac{C_1 C(x,y)}{|x-y|^{n-m\frac{n}{n+1}}}, \quad |x-y| < 1 \label{Knear}\\
| K(x,y) | & \leq \frac{C_2 C(x,y)}{|x-y|^{n+1}}, \quad |x-y| >1,\label{Kfar}
\end{align}
where $C(x,y) = \sup_{\xi \in \R^n, |\a| \leq n+1} (1+|\xi|)^{|\a| + m}| \p_\xi^\a a(x,y,\xi)| $ and $C_1$,$C_2>0$ are some absolute constants.
\end{lemma}
\begin{proof} In view of \eqref{Kper}, the lemma will follow from the corresponding estimates on $K_{\R^n}$. Let us fix an $R>1$ and the integer $p = n-m+1$. Let $\chi(\xi)$ be a standard cut-off function supported on the ball $|\xi|<2$, and let $\chi_R(\xi) = \chi(\xi/R)$. To prove \eqref{Knear} we write
\[
K_{\R^n}(x,y) = \int_{\R^n}  \chi_R(\xi) a(x,y,\xi) e^{i(x-y)\cdot \xi} d\xi +
\int_{\R^n} (1-  \chi_R(\xi) ) a(x,y,\xi) e^{i(x-y)\cdot \xi} d\xi .
\]
The first integral in simply bounded by $R^n \sup_{\xi}|a(x,y,\xi)| \leq R^n C(x,y)$. To estimate the second, let $j$ be such that $|x_j - y_j| \geq \frac{1}{\sqrt{n}} |x-y|$. We have
\[
\begin{split}
\int_{\R^n} (1-  \chi_R(\xi) ) a(x,y,\xi) e^{i(x-y)\cdot \xi} d\xi& = \frac{i^{-p}}{(x_j - y_j)^p}\int_{\R^n} (1-  \chi_R(\xi) ) a(x,y,\xi) \p_{\xi_j}^p e^{i(x-y)\cdot \xi} d\xi \\
&=
\frac{i^{-p}}{(x_j - y_j)^p}\int_{\R^n} (1-  \chi_R(\xi) ) a(x,y,\xi) \p_{\xi_j}^p e^{i(x-y)\cdot \xi} d\xi \\
&= \frac{(-i)^{-p}}{(x_j - y_j)^p}\int_{\R^n}\p_{\xi_j}^p( (1-  \chi_R(\xi) ) a(x,y,\xi) ) e^{i(x-y)\cdot \xi} d\xi .
\end{split}
\]
We have the bound
\[
\begin{split}
   |\p_{\xi_j}^p( (1-  \chi_R(\xi) ) a(x,y,\xi) ) | & \leq I_{R < |\xi| <2R} \sum_{l=0}^{p-1} \frac{1}{R^{p-l}} \frac{C(x,y)}{(1+|\xi|)^{l+m}} + I_{|\xi| \geq R} \frac{C(x,y)}{(1+|\xi|)^{p+m}} \\
   &\lesssim I_{|\xi| \geq R} \frac{C(x,y)}{(1+|\xi|)^{p+m}}.
\end{split}
\]
Thus,
\begin{equation}\label{e:trunc}
 \left| \int_{\R^n} (1-  \chi_R(\xi) ) a(x,y,\xi) e^{i(x-y)\cdot \xi} d\xi \right| \lesssim \frac{C(x,y)}{R |x-y|^p},
\end{equation}
and we obtain
\[
|K_{\R^n}(x,y)| \lesssim C(x,y) \left( R^n + \frac{1}{R|x-y|^{n+1-m}}\right).
\]
The minimum over $R$ is attained at $R = |x-y|^{-1 + \frac{m}{n+1}}$, and the bound \eqref{Knear} readily follows. 

The bound \eqref{Kfar} is obtained simply integrating by parts:
\[
|K_{\R^n}(x,y)| \lesssim \frac{1}{|x-y|^{n+1}} \int_{\R^n}| \p_{\xi_j}^{n+1} a(x,y,\xi)|d\xi 
\leq \frac{C_2 C(x,y)}{|x-y|^{n+1}}.
\]
\end{proof}
\begin{remark} Working with fractional derivatives one can reach a better bound
\[
| K(x,y) |  \leq \frac{C_\e C(x,y)}{|x-y|^{n-\frac{mn}{n+\e}}}, \quad |x-y| < 1 
\]
for any $\e>0$. It is more consistent with homogeneity of the amplitude. Furthermore, incorporating higher order derivatives in $C(x,y)$, one can show arbitrarily fast algebraic decay for $|x-y|>1$. We omit the details because bounds \eqref{Knear}, \eqref{Kfar} are sufficient for all our purposes. 
\end{remark}

 \subsection{Boundedness and compactness} The boundedness of $A$ as an operator from $L^2(\O)$ to $L^2(\O)$ in the classical case is well-understood and just as easily extends to the mixed case under question ( although most texts on PDOs treat either local $L^2$-boundedness or cases of left, right, or Weyl-quantized form of $A$). We will however revisit this issue again with the purpose to obtain a localization result for truncated operator (see \eqref{e:seeley}). This is needed later to determine the norm of $A$ in the Calkin algebra.
\begin{lemma}\label{l:bounded} If $a \in S^0$, then $A$ extends to a bounded operator  
from $L^2(\O)$ to $L^2(\O)$. Moreover, suppose $a=a(x,\xi)$ is $0$-homogeneous in $\xi$ for $|\xi|>1$. Let  $a_R = (1-\chi_R(\xi)) a(x,\xi)$. Then we have
\begin{equation}\label{e:seeley}
\limsup_{R\ra \infty} \|\op[a_R]\| \leq C \sup_{x\in \O, \xi \in \O^*} |a(x,\xi)|,
\end{equation}
for some absolute $C>0$.
\end{lemma}
\begin{proof} Let us partition $\O = \cup_{j=1}^\infty \O_j$ into disjoint boxes of side length $2\pi$, e.g. stacking with boxes $\T^d$. Let $1 = \sum_{j} \phi_j(y)$ be a partition of unity subordinate to the cover with sets $2 \O_j$. We can choose $\phi_j$ having uniformly bounded any number of derivatives. Let us split $A = A' + A''$, where $A'$ has amplitude $a'= a(x,y,\xi) \sum_{i,j: 2\O_i \cap 2\O_j \neq \emptyset}\phi_i(x) \phi_j(y)$, and $A''$ has $a''=a(x,y,\xi)\sum_{i,j: 2\O_i \cap 2\O_j = \emptyset}\phi_i(x) \phi_j(y)$. The kernel of $A''$ is thus supported off the diagonal, and hence in view of \eqref{Kfar} enjoys a convolution-type globally integrable majorant. Clearly $A''$ is bounded. Moreover, using \eqref{e:trunc} with $p=n+1$ for $a''$ shows that $\|\op[(1-\chi_R)a''\| \leq O(1/R) \ra 0$, as $R \ra \infty$.

As to $A'$ we observe that $a'$ is properly supported in the band $|x-y|<2$. Thus, $\supp A'u \ss \supp u + B_2$, where $B_r$ is the ball of radius $r$. By the finite intersection consideration, we have
\begin{equation}\label{e:finitinter}
\| A' u\|_{L^2}^2 \lesssim \sum_{i,j: 2\O_i \cap 2\O_j \neq \emptyset} \|\phi_i A (u\phi_j)\|_{L^2}^2.
\end{equation}
The amplitudes  $\phi_i(x) a(x,y,\xi) \phi_j(y)$ are supported on $(2\O_j+B_2 )\times 2\O_j \times \R^n$ with uniformly bounded derivatives in $(x,y)$. 
Thus, we have
\begin{equation}\label{}
|\widehat{\phi_i a \phi_j}^{x,y}(\eta_1, \eta_2,\xi)| \lesssim  \frac{ C_{n+1,n+1,0}(a)}{(1+|\eta_1|)^{n+1}(1+|\eta_2|)^{n+1}},
\end{equation}
for all $\eta_1,\eta_2,\xi \in \O^*$. Consider $u \in C^\infty_0(\O)$. By Plancherel,
\begin{equation}\label{}
\begin{split}
&\|\phi_i A (u\phi_j)\|_{L^2}^2 = \\
&=\int_{\O^*} \left| \int_{\O^*\times \O^*} \widehat{\phi_i a\phi_j}(\xi-\eta_1,\xi-\eta_2,\xi)\ \widehat{\rest{u}{2\O_j}}(\eta_2)\ d\eta_2\ d\xi  \right|^2 d \eta_1\\
&\lesssim C_{n+1,n+1,0}(a)   \int_{\O^*} \left| \int_{\O^*} \frac{1}{(1+|\xi-\eta_1|)^{n+1}} \int_{\O^*}\frac{1}{(1+|\xi-\eta_2|)^{n+1}}
|\widehat{\rest{u}{2\O_j}}(\eta_2)|\ d\eta_2\ d\xi  \right|^2 d \eta_1  \\
& =  C_{n+1,n+1,0}(a) \| h \star h \star \widehat{\rest{u}{2\O_j}} \|_{L^2}^2,
\end{split}
\end{equation}
 where $h(\xi) = \frac{1}{(1+|\xi|)^{n+1}}$, an integrable kernel. By Young,
\begin{equation}\label{e:Abound}
 \|\phi_i A (u\phi_j)\|_{L^2}^2 \lesssim C_{n+1,n+1,0}(a) \| \widehat{\rest{u}{2\O_j}} \|_{L^2}^2.
\end{equation}
Summing up over $i,j$ (where for each $j$ there are only finite number of $i$), we obtain $\| A u\|_{L^2} \lesssim C_{n+1,n+1,0}(a) \|u\|_{L^2}$. 

 To show \eqref{e:seeley}, since we already know that  $\|\op[(1-\chi_R)a''\|  \ra 0$, we can focus on the kernel $(1-\chi_R(\xi)) a'(x,\xi)$, or in view of \eqref{e:finitinter} only on  $(1-\chi_R(\xi))  \phi_i(x) a(x,\xi) \phi_j(y)$, for $i,j$ such that $2\O_i \cap 2\O_j \neq \emptyset$.

First, we truncate $a$ on the Fourier side in $x$. Let $\psi \in C^\infty_0(B_1)$ be nonnegative with $\phi(0) = 1$ and let $\psi_L(\xi) = \psi(\xi/L)$. Then $\f_L = \check{\psi_L}$ is a standard mollifier on $\O$. Consider $a_L(x,\xi) = (\f_L \star a(\cdot, \xi)) (x)$. Then, by regularity $C_{n+1,n+1,0}(a-a_L) \ra 0$ as $L\ra \infty$. Consequently, in view of \eqref{e:Abound}, $\| \op[(1-\chi_R) \phi_i ( a -a_L )\phi_j] \| \ra 0$ uniformly in $R$. We thus can focus on $\op[(1-\chi_R) \phi_i a_L \phi_j]$ only. Let us fix another scale $\d>0$. Again by regularity of the symbol $a_L$ in $\xi$ we know that 
\begin{equation}\label{xi-Lip}
\sup_{x\in \O} | \p_x^\a( a_L(x,\xi_1) - a_L(x,\xi_2))| \leq C \d, \text{ for all } \xi_1, \xi_2 \in \S \text{ with } |\xi_1 - \xi_2| <\d.
\end{equation}
 
We will now discritize $a_L$ is $\xi$ as follows. First, 
the unit sphere of $\R^n$ can be partitioned into a finite number of
tiles $T_1,\ldots,T_M$, $M = M(n,\d)$, of diameter less than $\d$,
so that any boundary point is shared by at most $2^{n-1}$ of the tiles. 
This can be achieved by slicing the cube $[-1,1]^n$ with the hyperplanes
$$
\{(x_1,\ldots,x_n): x_i= \d k, \quad k=-1/\d, \ldots ,1/\d\}.
$$
to produce a family of tiles on the surface of the
cube with the required intersection property. Then the radial projection onto the sphere yields with desired tiling. Let us fix a tag point $\xi_m \in T_m$ in each tile. Let us now consider the symbol 
\[
a_{L,\d} (x,\xi) = \sum_{m=1}^M a_L(x,\xi_m) I_{\{\xi/|\xi| \in T_m\}}.
\]
In view of \eqref{xi-Lip}, $C_{n+1,n+1,0}(a_L - a_{L,\d}) \lesssim \d$, implying similar bound on the operators by \eqref{e:Abound} uniformly . So, we reduce the problem to showing that for any fixed $L>0$ and $\d$, 
\begin{equation}\label{}
\limsup_{R \ra \infty} \| (1-\chi_R) \phi_i a_{L,\d} \phi_j] \|  \leq c \sup_{x\in \O, \xi \in \O^*} |a(x,\xi)|.
\end{equation}
So, let us fix $u \in L^2(2 \O_j)$. We can incorporate $\phi_j$ into $u$, since $\|\phi_j u \| \leq \|u\|$. We have
\begin{equation}\label{}
\begin{split}
\op[(1-\chi_R) \phi_i a_{L,\d}] u (x) & = \phi_i(x)\int_{\O^*} a_{L,\d}(x,\xi) (1- \chi_R(\xi)) \hat{u}(\xi) e^{i x \cdot \xi} d\xi \\
& = \phi_i(x) \sum_{m=1}^M a_L(x,\xi_m) \int_{\xi/|\xi| \in T_m}  (1- \chi_R(\xi)) \hat{u}(\xi) e^{i x \cdot \xi} d\xi.
\end{split}
\end{equation}
Denote 
\[
u_m (x) = \int_{\xi/|\xi| \in T_m}  (1- \chi_R(\xi)) \hat{u}(\xi) e^{i x \cdot \xi} d\xi.
\]
We know that the Fourier supports of $u_m$ are disjoint and lie in the region $|\xi| >R/2$. In particular, $\| \sum_m u_m \|_2^2 = \sum_m \|u_m\|_2^2 \leq \|u\|_2^2$. Recall that the Fourier supports of $a_L(x,\xi_m)$ belong to the fixed ball $B_L$. Thus, the sets 
\[
\supp( \widehat{a_L(\cdot,\xi_m)}\star \widehat{ u_m } ) \ss \supp( \widehat{a_L(\cdot,\xi_m)}) + \supp(\widehat{u_m}) \ss B_L + \supp(\widehat{u_m})
\]
have the $2^{n-1}$-fold intersection property in $m$, when $R$ is sufficiently large. So, by Plancherel,
\begin{equation}\label{}
\begin{split}
\| \op[(1-\chi_R) \phi_i a_{L,\d}] u \|_2^2 &\leq \| \sum_m a_L(\cdot,\xi_m) u_m \|_2^2 \leq c_n \sum_m \|a_L(\cdot,\xi_m) u_m \|_2^2 \\
& \leq c_n \|a_L(\cdot,\xi_m) \|_\infty^2 \sum_m \| u_m \|_2^2 \leq c_n \|a\|_\infty^2 \|u\|_2^2.
\end{split}
\end{equation}
This finishes the proof.
\end{proof}  
  
Compactness of PDOs on a non-compact domain such as mixed open space we consider is a subtle issue. This due to the fact that simply a decay of $a(x,y,\xi)$ as $\xi \ra \infty$ is not enough as it is in compact settings, see \cite{shubin,treves}. To regain compactness under this condition one has to insist on decay in spacial variables as well. 
We will be concerned only with the case $a \in S^m$, $m <0$. So, $A$ is an integral operator with integrable kernel $K$. Let us first state a general compactness condition for such operators. So, let us consider somewhat more general integral operator
\begin{equation}\label{Aint}
Au(x) = \int_{Y} K(x,y) u(y) dy,   x \in X,
\end{equation}
where $X,Y \ss \R^n$. If $I_Y = \sup_{y\in Y} \int_X |K(x,y)|dx$ and $I_Y = \sup_{x\in X} \int_Y |K(x,y)|dy$ are finite, then by interpolation $A: L^2(Y) \ra L^2(X)$ is bounded, and 
\begin{equation}\label{Anorm}
\|A\|_{L^2(Y) \ra L^2(X)} \leq \sqrt{ I_X I_Y }.
\end{equation}

\begin{lemma}\label{l:compact} Suppose $I_X, I_Y < \infty$ and the following conditions hold
\begin{description}
\item[(decay at infinity)] 
\begin{equation}\label{decay1}
\lim_{R \ra \infty} \sup_{|y|>R} \int_{X} |K(x,y)| dx \cdot \sup_{x\in X} \int_{|y|>R} |K(x,y)| dy = 0,
\end{equation}
and
\begin{equation}\label{decay2}
\lim_{R \ra \infty} \sup_{|x|>R} \int_{Y} |K(x,y)| dy \cdot \sup_{y\in Y } \int_{|x|>R} |K(x,y)| dx = 0.
\end{equation}
\item[(smoothness)] For any $R>0$,
\begin{equation}\label{smooth-y}
\lim_{|z| \ra 0} \sup_{|x|<R} \int_{|y|<R}| K(x,y+z) - K(x,y) | dy = 0,
\end{equation}
or 
\begin{equation}\label{smooth-x}
\lim_{|z| \ra 0} \sup_{|y|<R} \int_{|x|<R}| K(x+z,y) - K(x,y) | dy = 0.
\end{equation}
\end{description}
Then $A: L^2(Y) \ra L^2(X)$ is compact.
\end{lemma}
\begin{proof} Let $X'\ss X$ and $Y' \ss Y$ be two arbitrary bounded subdomains, and 
 $A': L^2(Y') \ra L^2(X')$ be the restriction/projection of $A$. Let us show that $A'$ is compact. Let $\f_\e$ is a standard supported on $B_1$. Suppose that \eqref{smooth-y} holds for $R$ large enough to engulf $X',Y'$. Consider the $y$-mollified kernel $K_\e(x,\cdot) = \f_\e \star K(x,\cdot)$, and $A'_\e$ the integral operator with kernel $K_\e$. Since $K_\e$ is bounded, $K_\e \in L^2(X'\times Y')$, and hence $A'_\e$ is a compact Hilbert-Schmidt operator. On the other hand, 
\[
\| A'_\e - A' \|^2_{L^2(Y') \ra L^2(X')} \leq 2I_Y \sup_{|x|<R, |z|<\e} \int_{|y|<R}| K(x,y+z) - K(x,y) | dy \ra 0,
\]
as $\e \ra 0$. If the second smoothness condition \eqref{smooth-x} holds, we apply the same argument to the dual operator $(A')^* : L^2(X') \ra L^2(Y')$, and compactness follows by duality.

Now with $\e>0$ fixed, we choose $R\gg 1$ so that the expressions in \eqref{decay1} and \eqref{decay2} are smaller than $\e$. Then in view of \eqref{Anorm}, the norms of $A$ as an operator $L^2(|y|>R, y\in Y) \ra L^2(X)$ and $L^2(Y) \ra L^2(|x| >R, x\in X)$ 
are less than $\e$. Yet $A : L^2(|y|<R, y\in Y)\ra L^2(|x| <R, x\in X)$ is compact by the above. This finishes the proof.
\end{proof}

As a consequence of \lem{Kbound} and \lem{l:compact} we obtain the following result.

\begin{theorem}\label{t:compact} Let $a \in S^{m}(\O)$ with $m<0$ and $A$ be the PDO on $\O$ given by \eqref{mixedPDO}. Suppose that for all $R>0$,
\begin{equation}\label{e:decay}
\lim_{x,y \ra \infty, |x-y|<R}  \sup_{\xi \in \R^n, |\a| \leq n+1} (1+|\xi|)^{|\a| + m}| \p_\xi^\a a(x,y,\xi)| =0.
\end{equation}
Then $A: L^2(\O) \ra L^2(\O)$ is compact.
\end{theorem}
\begin{proof} 
Recall that the kernel $K$ of $A$ satisfies the pointwise estimates of \lem{Kbound}, i.e. $K(x,y) \lesssim  C(x,y) F(x-y)$, where $F\in L^1(\O)$ is independent of $a$.
It remains to verify that our condition \eqref{e:decay} implies the hypotheses of \lem{l:compact}. Let us fix $R$ large. Then
\[
\int_{\O} | K(x,y) | dx \lesssim \int_{|x-y|<R} C(x,y) F(x-y) dx +  \sup_{x,y}C(x,y) \int_{|x-y|>R}F(x-y)dx.
\]
In view of \eqref{e:decay} in the limit as $y \ra \infty$, the right hand side does not exceed $\frac{1}{R}\sup_{x,y}C(x,y)$, which can be made arbitrarily small. Also, 
$\sup_{x\in \O} \int_{|y|>R} |K(x,y)| dy $ is of course bounded. So, \eqref{decay1} holds, and \eqref{decay2} is proved similarly. Finally, conditions \eqref{smooth-y}, \eqref{smooth-x} readily follow from the local smoothness $K\in C^\infty(x\neq y)$ and near diagonal integrability condition \eqref{Knear}.
\end{proof}
This finishes our general discussion of PDOs on mixed domains.

\section{Spectrum of the linearization}
The governing equations of a non-homogeneous ideal fluid are given by \eqref{e:RT}.
We assume the fluid is confined to the strip $\O = \T \times \R$, and is periodic in the first coordinate $x_1$. Let $u_0 = \lan U(x_2), 0 \ran$, $\r_0 = \r_0(x_2)$ be fixed. We assume 
\begin{equation}\label{dec}
U, \r_0 \in C^{3}(\R), \text{ and } U', U'', \r_0' , \r_0'' \ra 0 \text{ as } |x_2| \ra \infty,
\end{equation}
and 
\begin{equation}\label{e:novac}
\r_0>0, \quad \lim_{x_2 \ra \pm \infty} \r_0(x_2) = \r_{\pm} > 0.
\end{equation}
The pair $(u_0,\r_0)$ is a steady state solution to \eqref{e:RT} with the hydrostatic  pressure given by $\n p_0 =- g \r_0$.  The linearization around the steady state is given by \eqref{e:LRT}. We assume that the perturbation $(v,r)$ is periodic in $x_1$ direction, and has zero mean, 
\begin{equation}\label{e:mean0}
\int_\T v(x_1,x_2) \dx_1 = \int_\T r(x_1,x_2) \dx_1 = 0,
\end{equation}
for all $x_2\in \R$.  So, the Fourier support of perturbation belongs to $\O_0^* = \{(\xi_1, \xi_2): \xi_1 \in \Z \bs \{0\}, \xi_2 \in \R\}$, which remains away from the origin. Notice that this condition is preserved by the system \eqref{e:LRT}. In what follows, we will encounter amplitudes $a(x,y,\xi)$ depending on the vertical spacial coordinate only $a = a(x_2,y_2,\xi)$, and are globally $m$-homogeneous in $\xi$. The latter fact makes such an amplitude singular at the origin, but in view of the mean-zero condition \eqref{e:mean0} the PDO \eqref{mixedPDO} takes form 
\begin{equation}\label{mixedPDO2}
Au(x) = \int_{\O_0^*}\int_{\O} a(x_2,y_2,\xi)e^{i(x-y)\cdot \xi} u(y) dy d\xi,
\end{equation}
which restricts to the region $|\xi| \geq 1$ in the outer integral. This allows us to replace the amplitude with $(1-\chi(\xi)) a(x_2,y_2,\xi)$ without changing the action of the operator making the new amplitude locally smooth and of proper class $S^m$. We assume from now on that such modification has been made every time we consider an operator  \eqref{mixedPDO2} without altering notation for the amplitude. 

Since $a$ depends on the vertical coordinate, and integration in $\xi_2$ is continuous, one can  integrate by parts and find that the new amplitude 
\begin{equation}\label{left-q}
a(x,x,\xi) + \int_0^1 \p_{\xi_2} \p_{y_2} a(x, sx+ (1-s)y, \xi) ds = a_0(x,\xi) + r(x,y,\xi).
\end{equation}
defines the same operator. 

\subsection{Recovery of the pressure}
Since $v_t + u_0 \cdot \n v - v \cdot \n u_0$ is divergence-free, taking divergence of the momentum equation, we obtain
\begin{equation}\label{orig-div}
2 \diver( v\cdot \n u_0) + \diver(\r_0^{-1} \n q) = \diver(\vec{g} \r_0^{-1} r).
\end{equation}
Since the form $\int_\O \r_0^{-1} \n q_1 \cdot \n q_2 \dx$ is coercive on $H^1(\O)$ subject to the horizontal mean zero condition, by the Riesz representation theorem, we have a well-defined linear bounded map $Q(v,r) = \n q$ from $L^2 \times L^2$ to $L^2$.  We now in a position to track down the principal symbol of this map with compactness control on the remainder operators. We have
\[
2\r_0 \diver( v\cdot \n u_0) + \D q  - \r_0^{-1} \n \r_0 \cdot \n q = - g \p_{2} r + g \r_0^{-1} \r_0' r.
\]
So,
\[
\n q = - 2\n \D^{-1} \r_0 \diver( v\cdot \n u_0) + \n \D^{-1} \r_0^{-1} \n \r_0 \cdot \n q
 - g \n \D^{-1} \p_{2} r + g \n \D^{-1} \r_0^{-1} \r_0' r
\]
Now, $\n \D^{-1} \r_0^{-1} \n \r_0 \cdot \n q$ is a composition of $Q$ and a PDO with the amplitude $a(x,y,\xi) =  \frac{\xi}{|\xi|^2} \frac{\n \r_0(y) }{\r_0(y)}$. It clearly satisfies the hypothesis of our \lem{l:compact}. So, this term contributes a compact operator. For the same reason the map $r \ra g \n \D^{-1} \r_0^{-1} \r_0' r$  is compact too. We have
\[
\n q = - 2\n \D^{-1} \r_0 \diver( v\cdot \n u_0) - g \n \D^{-1} \p_{2} r  + B_1(v,r),
\]
where $B_1: L^2(\O) \times L^2(\O) \ra L^2(\O)$ is compact. Next, $\n \D^{-1} \r_0 = \r_0 \n \D^{-1} + B_2$, where according to \eqref{left-q}  $B_2$ has amplitude $a_2(x,y,\xi) = \p_{\xi_2}(\xi |\xi|^{-2}) \int_0^1 \r_0'(sx + (1-s)y) ds$. So, we have
\begin{equation}\label{}
\begin{split}
B_2( \diver( v\cdot \n u_0) )(x) & = \int_{\O_0^*}\int_{\O} a_2(x,y,\xi)e^{i(x-y)\cdot \xi} \diver_y( v\cdot \n u_0)(y) dy d\xi \\
& =  - \int_{\O_0^*}\int_{\O} \n_y a_2(x,y,\xi)  \p u_0(y) v(y) e^{i(x-y)\cdot \xi}dy d\xi \\
&+
\int_{\O_0^*}\int_{\O} a_2(x,y,\xi) \xi \cdot   \p u_0(y) v(y) e^{i(x-y)\cdot \xi}dy d\xi.
\end{split}
\end{equation}
Under the assumption \eqref{dec} the two amplitudes $\n_y a_2(x,y,\xi)  \p u_0(y)$ and $a_2(x,y,\xi) \xi \cdot   \p u_0(y) $ clearly satisfy \lem{l:compact}. We thus obtain
\[
\r_0^{-1} \n q = - 2 \n \D^{-1} \diver( v\cdot \n u_0) - g \r_0^{-1} \n \D^{-1} \p_{2} r  + B_3(v,r),
\]
where $B_3$ is compact. Finally we  left-quantize the first PDO by replacing its amplitude $-2|\xi|^{-2} \xi \otimes \xi \p u_0(y)$ with (according to \eqref{left-q})
\[ 
-2|\xi|^{-2} \xi \otimes \xi \p u_0(x) -2 \p_{\xi_2}(|\xi|^{-2} \xi \otimes \xi ) \int_0^1 \p u'_0(s x+(1-s) y) ds.
\]
The latter is a symbol of class $S^{-1}$ decaying as $x$ and $y$ go to $+\infty$ or $-\infty$ simultaneously, which is sufficient for \lem{l:compact}. It thus contributes a compact term. Putting all these together we can rewrite our original system \eqref{e:LRT} in the advective form
\begin{equation}\label{e:adv}
V_t + u_0 \cdot \n V = A_0 (V) + B(V),
\end{equation}
where $V = \bigl(\begin{smallmatrix}
v \\ r
\end{smallmatrix} \bigr)$ is the state variable, $B$ is compact, and $A_0$ is left-quantized PDO with a matrix symbol given by
\[
a_0(x,\xi) = \left[\begin{array}{cc}- \p u_0(x) + 2 \frac{\xi \otimes \xi}{|\xi|^2} \p u_0(x)  &  \frac{g}{\r_0(x)} \frac{\xi_1 \xi^{\perp}}{|\xi|^2}\\ \n \r_0(x) & 0\end{array}\right]: \C^3 \ra \C^3,
\]
where $\xi^{\perp} = \lan \xi_2, - \xi_1 \ran$. The system is subject to constraints
\begin{equation}\label{constraint}
\diver v = 0, \quad \int_\T V(x_1,x_2) d x_1 = 0.
\end{equation}

\subsection{Microlocal structure of the semigroup}
Let $G_t : \Ldiv \ra \Ldiv$ be the semigroup (in fact group) generated by the system \eqref{e:adv}. Here $\Ldiv$ stands for the space of $L^2$-integrable fields $V$ satisfying \eqref{constraint}.  Let $\cC$ denote the Calkin algebra over $\Ldiv$, i.e. the space of bounded operators over $\Ldiv$ modulo compact, endowed with the natural factor-norm. By definition, $r_{\ess}(G_t, \Ldiv)$ is the spectral radius of $G_t$ as an element of the $C^*$-algebra $\cC$.  By the classical Nussbaum Theorem, \cite{nuss}, $r_{\ess}$ coincides with the radius of the Browder spectrum as well as the Fredholm spectral radius.

Similar to the periodic case considered in \cite{S-ess} we seek to describe the PDO structure of $G_t$ up to a compact operator. As in \cite{S-ess} one expects that the essential dynamics of \eqref{e:LRT} is governed by the finite dimensional dynamical system given by
\begin{equation}\label{bas}
\left\{\begin{split}
x_t & = u_0(x) \\
\xi_t & = - \p u_0^\top(x) \xi \\
b_t &= a_0(x,\xi)b.
\end{split}\right.
\end{equation}
To recall, \eqref{bas} represents the leading order dynamics written in Lagrangian coordinates of the evolution of the WKB ansatz
\[
V(x,t) = b(x,t) e^{i S(x,t)/\e} + O(\e),
\]
where $\xi = \n_x S$ and $(b_1,b_2) \cdot \xi = 0$. The first two equations of \eqref{bas} determine the flow on the cotangent bundle of $\O$, i.e. $T^*\O = \O \times \R^2$. With our stratification  case, the flow is explicitly given by 
\begin{equation}\label{}
\chi_t(x_1,x_2;\xi_1,\xi_2) = (U(x_2) t +x_1, x_2; \xi_1, - U'(x_2) \xi_1 t+ \xi_2).
\end{equation}
The $b$-equation is the non-autonomous system of the flow $\chi_t$:
\begin{equation}\label{ODE0}
b_t = a_0(\chi_t(x_0,\xi_0))b.
\end{equation}
Notice that the orthogonality condition $(b_1,b_2) \cdot \xi = 0$ is preserved. So, the system \eqref{bas} generates a cocycle (fundamental solution) $B_t(x,\xi)$ over the flow $\chi_t$ acting on the fiber bundle $\cF$ over $ \O \times (\R^2 \bs \{0\})$ with fibers given by $F(x,\xi)= F(\xi) = \{ b\in \C^3: (b_1,b_2) \cdot \xi = 0\}$. Note that $\cF$ is smooth. It will significantly simplify the arguments to view $B_t$ as a restriction to $\cF$ of a ``free" cocycle, also denoted $B_t$, obtained by considering the ODE \eqref{ODE0} with unrestricted initial condition $b_0 \in \C^3$. This way, $B_t(x,\xi)$ can be viewed as a $(x,\xi)$-dependent $3\times 3$ matrix, for which we can make sense of partial derivatives without resorting to covariant differentiation. Now, it follows directly from the form of the symbol $a_0$ and the classical ODE theory that $\{B_t(x,\xi)\}_{t \geq 0, x\in \O, \xi \neq 0}$ is smooth, $0$-homogeneous in $\xi$, depends spacialy only on $x_2$ coordinate, and remains uniformly smooth on any finite time  interval, i.e. for all $\a,\b$, $\sup_{t<T, x\in \O, \xi \in \S} \| \p_x^\a \p_\xi^\b B_t(x,\xi)\| < \infty$. In particular, it defines a matrix symbol of class $S^0(\O)$. We need to establish stabilization at infinity to conclude that the symbol $\p_\xi \p_{x_2} B_t(x,\xi)$ defines a compact operator. It will be a consequence of \thm{t:compact} and the following.
\begin{lemma}\label{l:stab}
For every $\b$ and $T$, one has
\begin{equation}\label{}
\lim_{x_2 \ra \pm \infty} \sup_{0<t<T, \xi \in \S} \| \p_\xi^\b \p_{x_2} B_t(x,\xi) \| = 0.
\end{equation}
\end{lemma}
\begin{proof}
Let us assume first $\b = 0$, and  differentiate the $b$-equation in $x_2$:
\begin{equation}\label{}
\p_t \p_{x_2} B_t(x,\xi) = \p_{x_2}(a_0(\chi_t(x,\xi))) B_t(x,\xi) + a_0(\chi_t(x,\xi))\p_{x_2}B_t(x,\xi).
\end{equation}
Since, clearly $\p_{x_2} B_0(x,\xi) = \p_{x_2} Id =0$, by Duhamel principle we obtain
\begin{equation}\label{b-duh}
\p_{x_2} B_t(x,\xi) = \int_0^t B_{t-s}( \chi_s(x,\xi))\p_{x_2}(a_0(\chi_s(x,\xi))) B_s(x,\xi)ds.
\end{equation}
Since by uniqueness $\inf_{t<T, \xi \in \S} |\xi(t)| > 0$, a routine computation shows that 
\[
\sup_{s<T, \xi \in \S} |\p_{x_2}(a_0(\chi_s(x,\xi))) | \lesssim |\r_0'(x_2)|+|\r_0''(x_2)|+|U'(x_2)|+|U''(x_2)| \ra 0.
\]
This establishes \lem{l:stab} for $\b = 0$. The case $|\b|>0$ goes by induction, differentiating \eqref{b-duh} and using the fact that for any $\b$, $\sup_{s<T, \xi \in \S} |\p_\xi^\b \p_{x_2}(a_0(\chi_s(x,\xi))) |$ is controlled by the same $ |\r_0'(x_2)|+|\r_0''(x_2)|+|U'(x_2)|+|U''(x_2)| $. 

\end{proof}

Let $\pi(\xi): \C^3 \ra F(\xi)$ be the orthogonal projection. Note that $\pi \in S^0$. Let $\Pi = \op[\pi] :L^2(\O) \ra \Ldiv$ also be the orthogonal projection. Here as before $L^2(\O)$ denotes the space with zero mean condition in horizontal direction.
Let $\f_t : \O \ra \O$ be the integral flow of $u_0$, i.e. $\f_t(x_1,x_2) = (U(x_2)t + x_1, x_2)$.

\begin{lemma}\label{l:str} For each $t$ the action of the semigroup $G_t$ is given by 
\begin{equation}\label{e:micro}
G_t u = \Pi [ (\op[B_t] u) \circ \f_{-t} ] + K_t u,
\end{equation}
where $K_t$ is a compact operator on $\Ldiv$, and 
\[
\op[B_t] u(x) = \int_{\O^*_0} B_t(x,\xi) \hat{u}(\xi) e^{i \xi \cdot x} d \xi.
\]
\end{lemma}

Let us make some preliminary observations. First, the system with eliminated pressure \eqref{e:adv} defines a semigroup on the whole space $L^2(\O)$ too. Indeed, the advection $u_0 \cdot \n$ is clearly a generator while the rest is a bounded operator given by a combination of explicit PDOs and an implicit operation $Q(v,r) = \n q$ which solves \eqref{orig-div} boundedly even if the input $v$ is not divergence-free. This semigroup, also denoted $G_t$ with a little abuse, leaves $\Ldiv$ invariant and of course the restriction restores the original semigroup (see discussions in \cite{}). If we prove the formula 
\begin{equation}\label{l:free}
G_t u = (\op[B_t] u) \circ \f_{-t}  + K_t u,
\end{equation}
on the entire space $L^2(\O)$, where $G_t$ and $B_t$ are the extentions, then by restricting it and projecting it to $\Ldiv$ will give the required \eqref{e:micro}.

Second, we will need the composition and change of variables formulas recast in our mixed domain settings with compactness check on remainders. Let $a(x_2,\xi) \in S^0$ be a matrix symbol that depends only on $x_2$. Then we have
\begin{equation}\label{e:comp}
\op[a] \circ \op[B_t] = \op[a \circ B_t] + \comp.
\end{equation}
(Note that the usual proper support assumption is not necessary here as we are not seeking full asymptotic expantion of the composition symbol, and the composition makes sense as both operators are bounded on $L^2$). Indeed, according to \eqref{left-q} another amplitude of $\op[B_t] $ is given by $B_t(y,\xi) - \int_0^1 \p_{\xi_2}\p_{y_2} B_t(s x_2 + (1-s) y_2,\xi) ds$. It follows from \lem{l:stab} and \lem{l:compact} that $\op[B_t(x,\xi)]$ differs from $\op[B_t(y,\xi)]$ by a compact operator. On the other hand,
\[
\op[a(x,\xi)] \circ \op[B_t(y,\xi)] u(x) = \int_{\O^*} \int_{\O} a(x,\xi) B_t(y,\xi) u(y) e^{i (x-y)\cdot \xi} dy d\xi.
\]
But then again by \eqref{left-q}, the amplitude $a(x,\xi) B_t(y,\xi)$ can be replaced with 
\[
a(x,\xi) B_t(x,\xi) +  \int_0^1 \p_{\xi_2}(a(x_2,\xi) \p_{y_2} B_t(s x_2 + (1-s) y_2,\xi)) ds.
\]
The symbol under the integral again gives a compact operator and \eqref{e:comp} is proved. And finally, the change of variable formula:
\begin{equation}\label{e:ch}
(\op[a \circ \chi_t] u) \circ \f_{-t} = \op[a]( u \circ \f_{-t} ) + \comp.
\end{equation}
Indeed, making routine change of variables on the right hand side we obtain
\[
(\op[a \circ \chi_t] u) \f_{-t} = \op[\tilde{a}] ( u \circ \f_{-t} ),
\]
where 
\[
\tilde{a}(x,y,\xi) = a\left(x, \p^\top \f_{-t}(x) \left( \int_0^1 \p^\top \f_{-t}( s y + (1-s)x) ds \right)^{-1} \xi \right),
\]
or more explicitely,
\[
\tilde{a}(x,y,\xi) = a\left(x_2,\xi_1,t\xi_1 \left( \int_0^1 U'(s y_2 + (1-s) x_2 )ds - U'(x_2) \right)+\xi_2 \right).
\]
In view of \eqref{left-q} we can replace this amplitude with 
\begin{equation*}\label{}
\begin{split}
a(x,\xi) + \int_0^1\int_0^1 \p_{\xi_2}^2 a\left(x_2,\xi_1, t\xi_1 \left( \int_0^1 U'(s (\tau x_2 +(1-\tau)y_2) + (1-s) x_2 )ds - U'(x_2) \right)+\xi_2 \right) \\ t s\xi_1  U''(s (\tau x_2 +(1-\tau)y_2) + (1-s) x_2 ) ds d\tau.
\end{split}
\end{equation*}
The latter identifies a symbol in class $S^{-1}$ that satisfies \lem{l:compact}. This proves \eqref{e:ch}.
\begin{proof}[Proof of \eqref{l:free} and hence \lem{l:str}]
Let $v_0 $ be in the domain the generator of \eqref{e:adv}. Let $u(t,x) = (\op[B_t] v_0 )(\f_{-t}(x))$. Then $u(t,\f_t(x))$ satisfies
\[
(u_t+u_0 \cdot \n u)(t,\f_t(x)) = \op[ (a_0 \circ \chi_t) B_t]v_0 (x),
\]
as follows from \eqref{ODE0}. So, in view of \eqref{e:comp} and \eqref{e:ch},
\begin{equation}\label{}
\begin{split}
u_t+u_0 \cdot \n u  & = ( \op[ (a_0 \circ \chi_t) B_t] v_0 )\circ \f_{-t}  \\
& = ( \op[a_0 \circ \chi_t] \circ \op[B_t] v_0 )\circ \f_{-t}  +\comp_t(v_0) \\
& = \op[a_0](   (\op[B_t] v_0 )\circ \f_{-t} ) +\comp_t(v_0) \\
& = \op[a_0] u +\comp_t(v_0).
\end{split}
\end{equation}
We see that $u$ satisfies the same equation as $G_t v_0$ up to a compact perturbation. By Duhamel's principle, this implies \eqref{l:free}.
\end{proof}

Before we proceed, let us discuss asymptotic behavior of $B_t$ at infinity. Let us notice that as $x_2 \ra \pm \infty$, the principal symbol $a_0$ converges to 
$ \bigl[\begin{smallmatrix}
 0 & g \r_{\pm \infty}^{-1} \xi_1 \xi^\perp |\xi|^{-2}\\ 0 & 0
\end{smallmatrix} \bigr]$, whose cocycle is explicit. So, let us fix a monotone positive function $\bar{\r}$ so that $\lim_{x_2 \ra \pm \infty} \bar{\r}(x_2) = \r_{\pm \infty}$ and $\bar{\r}', \bar{\r}'' \ra 0$ at infinity. Let us consider the new symbol
\[
\bar{a}(x,\xi) =  \left[\begin{array}{cc}0 &  \frac{g}{\bar{\r}(x_2)} \frac{\xi_1 \xi^{\perp}}{|\xi|^2}\\ 0 & 0\end{array}\right].
\]
It generates the cocycle
\[
\bar{B}_t(x,\xi) = \left[\begin{array}{cc} \Id &  \frac{g}{\bar{\r}(x_2)} \frac{\xi_1 \xi^{\perp}}{|\xi|^2} t \\ 0 & \Id\end{array}\right]
\]
over  the stationary phase flow $(x,\xi) \ra \chi_t(x,\xi)$. 

\begin{lemma}\label{l:asym}
We have for all $\a$
\begin{equation}\label{e:asym}
\lim_{x_2 \ra \pm \infty} \sup_{\xi \in \S} \| \p^\a_\xi( B_t(x,\xi) - \bar{B}_t(x,\xi) ) \| = 0.
\end{equation}
\end{lemma}
\begin{proof}
We have
\[
\p_t(B_t - \bar{B}_t) = a_0(\chi_t(x,\xi))(B_t- \bar{B}_t) + (a_0(\chi_t(x,\xi)) - \bar{a}(x,\xi)) \bar{B}_t.
\]
Since $B_0 - \bar{B}_0 = 0$, by Duhamel,
\begin{equation}\label{diffB}
B_t - \bar{B}_t = \int_0^t B_{t-s}(\chi_s(x,\xi)) (a_0(\chi_s(x,\xi)) - \bar{a}(x,\xi)) \bar{B}_s(x,\xi) ds.
\end{equation}
We have 
\[
a_0(\chi_s(x,\xi)) - \bar{a}(x,\xi) = a_0(x_2, \p \f_s^{-\top}(x)\xi) - \bar{a}(x_2,\xi),
\]
and $\p \f_s^{-\top}(x) \ra \Id$ as $x_2 \ra \pm \infty$ uniformly on $0\leq s\leq t$.
On the other hand $\p^\a_\xi a_0(x,\xi) \ra \p^\a_\xi \bar{a}(x,\xi)$ uniformly in $\xi \in \S$. Thus,
\[
\p^\a_\xi (a_0(x_2, \p\f_s^{-\top}(x)\xi) - \bar{a}(x_2,\xi)) = \p^\a_\xi a_0(x_2, \p \f_s^{-\top}(x)\xi) (\p \f_s^{-\top}(x))^\a - \p^\a_\xi \bar{a}(x_2,\xi) \ra 0
\]
uniformly as desired. Applying $\p_\xi^\a$ to \eqref{diffB} we obtain \eqref{e:asym}.
\end{proof}

\begin{lemma}\label{l:norm}
For each $t$ we have
\begin{equation}\label{}
\sup_{x\in \O, \xi \in \S} \|B_t(x,\xi)\|_{F(\xi) \ra F(\xi)} \leq \|G_t\|_\cC \leq C\sup_{x\in \O, \xi \in \S} \|B_t(x,\xi)\|_{F(\xi) \ra F(\xi)},
\end{equation}
where $C>0$ is an absolute constant.
\end{lemma}

\begin{proof} The bound from below is relatively easy to show. Let us first observe that in view of \eqref{e:comp} and \eqref{e:ch},
\[
G_t u = (\op[\pi \circ \chi_{t}] \op[B_t] u) \circ \f_{-t} + \comp = (\op[B_t]u)\circ \f_{-t} + \comp.
\]
(In the latter the operators are to be understood as $\Ldiv \ra L^2(\O)$, but the sum is valued in $\Ldiv$). So, let $\e>0$ be arbitrary, and let $x_0 \in \O$, $\xi_0 \in \S$, $b_0 \in F(\xi_0)$, $|b_0| = 1$  be such that $| B_t(x_0,\xi_0) b_0 | > M - \e$, $M = \sup_{x\in \O, \xi \in \S} \|B_t(x,\xi)\|_{F(\xi) \ra F(\xi)} $. Let $U_0$ be an open neighborhood of $x_0$ in which $| B_t(x,\xi_0) b_0 | > M - \e$ for all $x\in U_0$. Let us fix $h\in L^2(\O)\cap C^\infty_0(U_0)$ with $\|h\|_2 = 1$, and consider sequence $u_\d = \Pi( b_0 h (x) e^{i \xi_0 \cdot x /\d})$. By the classical localization principle (note that $h$ is compactly supported, see \cite{shubin}) we have $u_\d =  b_0 h (x) e^{i \xi_0 \cdot x /\d} + O(\d)$, as $\d \ra 0$, and certainly $u_\d \ra 0$ weakly in $L^2$. Again, by localization and compactness of remainders,
\[
G_t u_\d(x) = B_t(\f_{-t}(x), \xi_0) b_0 h (\f_{-t}(x)) e^{i \xi_0 \cdot \f_{-t}(x) /\d} + o(1).
\]
So, $\|G_t u_\d\|_{L^2} \geq \inf_{x\in U_0}| B_t(x,\xi_0) b_0 | - o(1)> M - \e - o(1)$. This establishes the lower bound.

Now, let us fix a large $R>0$, and consider a smooth rescaled cut-off function $\phi_R(\xi) = \phi(\xi/R)$. By \lem{l:bounded} the norm of $\op[(1-\phi_R) B_t]$ for large $R$ enjoys the necessary bound. On the other hand, by  \lem{l:asym}, $\op[ \phi_R B_t ] = \op [ \phi_R \bar{B}_t]$ up to a compact operator. Yet,
\[
 \|  \op [ \phi_R \bar{B}_t] \|_\cC \leq \|  \op [ \phi_R \bar{B}_t] \|_{L^2 \ra L^2} \leq C \left(1+ \frac{g}{\min\{\r_\infty, \r_{-\infty}\}}\right),
\]
where in the latter inequality we used explicit form of the symbol $ \bar{B}_t$ and monotonicity of $\bar{\r}$. Clearly, $1+ \frac{g}{\min\{\r_\infty, \r_{-\infty}\}} \leq M$. So, this term is also under control.
\end{proof}

Let us denote by $\mu$ the classical Lyapunov-Oseledets exponent of the b-cocycle
\begin{equation}\label{}
\mu = \lim_{t \ra \infty} \frac{1}{t} \log \sup_{x\in \O, \xi \in \S} \|B_t(x,\xi)\|_{F(\xi) \ra F(\xi)}.
\end{equation}
As a consequence of \lem{l:norm} and the general results of the spectral theory we obtain the desired result.

\begin{theorem}\label{t:main}
The essential spectral radius of the semigroup operator at each $t$ is given by
\begin{equation}\label{}
r_{\ess}(G_t, \Ldiv) = e^{t \mu}.
\end{equation}
\end{theorem}

We briefly discuss an extension of \thm{t:main} to the case of Sobolev spaces. 
Let us fix a smoothness parameter $m \in \R$, not necessarily an integer, and consider the corresponding Sobolev space
\begin{equation}\label{e:sobolev}
\Hdiv = \{ u\in H^m: \hat{u}(\xi) \in F(\xi), \text{ for all } \xi \in \O_0^*\}.
\end{equation}
The equation \eqref{e:adv} generates a group $\{G^m_t\}_{t\in \R}$ on $\Hdiv$ as well. We follow the argument of \cite[Section 3.5]{S-ess} to reduce the Sobolev case back to $\Ldiv$. We consider the isomorphism $I_m = |\n|^m : \Hdiv \ra \Ldiv$. Let $L_m$ denote the generator of the group $\{G^m_t\}_{t\in \R}$ on $\Hdiv$. Consider
a new generator $L_0 = I_m L_m I_m^{-1}$ on $\Ldiv$. Clearly, if $G^0_t = e^{tL_0}$, then $G^0_t = I_m G_t^m I^{-1}_m$, and so the spectrum of $G^0_t$ on $\Ldiv$ is equal to the spectrum of $G_t^m$ on $\Hdiv$.  The technique developed above allows to apply the classical composition formula for symbols, giving  the operator $L$ the same advective structure as the original PDE \eqref{e:adv} with its principal symbol given by
\[
a_m(x,\xi) = a_0(x,\xi) - m (\p u_0^{\top}(x) \xi, \xi)
|\xi|^{-2} \Id.
\]
The corresponding $b$-cocycle $B^m$ is given by
\begin{equation}\label{BXcocycle}
B^m_t(x,\xi) = \left|\p \f_t^{-\top}(x) \xi \right|^m
B_t(x,\xi),
\end{equation}
where $B_t(x,\xi)$ is the original cocycle. The analysis above therefore applies to the semigroup with leading order PDO given by the new cocycle $B_t^m$ as its symbol, and hence we conclude the formula for $G_t^m$:
\begin{equation}\label{}
r_{\ess}(G_t, \Hdiv) = e^{t \mu_m}.
\end{equation}
where 
\begin{equation}\label{}
\mu_m = \lim_{t \ra \infty} \frac{1}{t} \log \sup_{x\in \O, \xi \in \S} \|B^m_t(x,\xi)\|_{F(\xi) \ra F(\xi)}.
\end{equation}
However, the flow $\f_t$ is linear, and hence, it contributes no term of non-trivial  exponential type.  This readily implies that $\mu_m = \mu$ for all $m\in \R$.

\begin{corollary}\label{c:main}
    The essential spectral radius of the semigroup on $\Hdiv$ is given by
    \begin{equation}\label{}
    r_{\ess}(G_t, \Hdiv) = e^{t \mu}.
    \end{equation}
\end{corollary}

\subsection{Formula for $\mu$} Let us now compute the value of $\mu$ explicitly. First recall, from the basic cocycle theory that $\mu$ is determined by the local Lyapunov exponents:
\[
\mu =\sup_{x\in \O, \xi \in \S, b_0 \in F(\xi), |b_0|=1} \lim_{t \ra \infty} \frac{1}{t} \log  |B_t(x,\xi)b_0|.
\]
So, it suffices to investigate growth of individual solutions to  \eqref{ODE0}.
Let us write $b = (b_1,b_2,r)$. The system \eqref{ODE0} in components reads
\begin{equation}\label{}
\left\{\begin{split}
\dot{b}_1 & = - U'(x_2)b_2 + 2 U'(x_2) \frac{\xi_1^2}{|\xi|^2} + \frac{g \xi_1 \xi_2}{\r_0 |\xi|^2} r \\
\dot{b}_2 & = 2 \frac{\xi_1\xi_2}{|\xi|^2} U'(x_2) b_2 - \frac{g \xi_1^2}{\r_0 |\xi|^2}r \\
\dot{r} & = - b_2 \r_0'(x_2).
\end{split}\right.
\end{equation}
Eliminating $r$ from the second equation, we obtain the following ODE for $u = b_2$:
\[
\ddot{u} - 4 \frac{\xi_1\xi_2}{|\xi|^2}U'(x_2) \dot{u} + \left[ 4(U'(x_2))^2 - \frac{g \r_0'}{\r_0} \right] \frac{\xi_1^2}{|\xi|^2} u = 0.
\]
Making the change
\begin{equation}\label{e:change}
v = \exp\left\{ 2\xi_1 U'(x_2) \int \frac{\xi_2}{|\xi|^2} dt \right\} u
\end{equation}
we obtain the Sturm-Liouville problem
\begin{equation}\label{e:SL}
\ddot{v} + p(t) v = 0,
\end{equation}
where 
\[
p(t) = \left[ 4(U'(x_2))^2 - \frac{g \r_0'}{\r_0} \right] \frac{\xi_1^2}{|\xi|^2} +2 \xi_1 U'(x_2) \left(\frac{\xi_2}{ |\xi|^{2}}\right)_t - 4 \left[\frac{\xi_1\xi_2}{|\xi|^2}U'(x_2)\right]^2.
\]
Suppose that $U'(x_2) = 0$. Then $v = u$, $\xi$ is stationary, and we have
\[
\ddot{u} - \frac{g \r_0'}{\r_0}\frac{\xi_1^2}{|\xi|^2} u = 0.
\]
It now depends on the sign of $\r_0'(x_0)$. If it is positive, which corresponds to the Rayleigh-Taylor unstable regime, we obtain exponentially growing solutions with maximum exponent given by $\sqrt{ \frac{g \r_0'}{\r_0} }$. Otherwise we have oscillatory or linear solutions with $0$ exponential growth. 

Suppose $U'(x_2) \neq 0$. In the case $\xi_1 =0$ we again obtain linear solutions. Otherwise, let us notice that 
\[
\lim_{t \ra \infty} t^2 p(t) = 2 - \frac{g \r_0'}{\r_0 (U')^2} = p_0.
\]
So, comparing solutions of \eqref{e:SL} with the classical Euler equation $v'' + \frac{p_0}{t^2} v = 0$ we conclude that $v$ has only polynomial growth. Hence, so does $u$ since the exponent in \eqref{e:change} contributes a polynomial term too. We see that in this case the local Lyapunov exponent vanishes. We thus arrive at the following formula
\begin{equation}\label{e:mu2}
\mu = \max\left\{0, \sup_{x: U'(x) = 0, \r'_0(x)>0} \sqrt{   \frac{g \r_0'(x)}{\r_0(x)} } \right\}.
\end{equation}
Let us make two observations. When $\r_0'(x) \leq 0$ at all critical points of the shear $U$, it is clear that $\mu=0$. Thus, variations in the background velocity profile changes the character of the Rayleigh-Taylor instability from shortwave to possibly longwave. If $U \equiv 0$, then $\mu >0$ whenever there is a point where $\r_0'>0$ and thus Rayleigh-Taylor instability takes a shortwave nature. However as shown in Guo and Hwang \cite{HG}, this exponent $\mu$ is a limit from below of a sequence of exact eigenvalues as  well, and moreover, $\mu$ is the exponential rate of the semigroup. Thus, there is no spectrum beyond $\mu$ in this case. In fact, \cite{HG} uses 
a variational formula for the exponent:
\begin{equation}\label{}
\L^2 = \sup_{v\in L^2(\O)} \frac{ \int g \r'_0 v^2 dx}{\int \r_0 v^2 dx}.
\end{equation}
But, relabeling $f = \r_0 v^2$, we see that the above sup becomes
\begin{equation}\label{}
\L^2 = \sup_{\| f\|_{L^1(\O) } = 1, f \geq 0} \int \frac{g \r'_0}{\r_0} f dx = \mu^2.
\end{equation}
So, the two exponents coincide in the case when $U=0$.

\subsection{Unstable states with neutral essential spectrum}

Let us fix a steady state with a shear $U' \neq 0$ at any point where $\r_0'>0$, thus making the essential spectrum neutral. Let us form the classical ansatz

Let us note that \eqref{e:Rayl} is the full linearized system.

\begin{theorem}\label{t:Rayl} Suppose $U'(y) \neq 0$ and $\rho_0'(y)>0$ for all $y\in \R$, and let all the assumptions \eqref{dec}, \eqref{e:novac} be satisfied. Then for every $k\in \N$ there exists an $\e_0>0$ such that for every $|\e|<\e_0$ the state $(\e U, \rho_0)$ has a solution $(c,\phi,r)$ with $\im c > 0$, and $r,\phi \in H^m(\R)$ for all $m>0$.
\end{theorem}
\begin{proof} 
Let us recall that for $\e =0$ there exists an unstable solution to \eqref{e:Rayl} for every $k>0$ as shown in \cite{HG}. Namely, it results through solving the variational problem, $c_k =  i\frac{\l_k}{k}$, where 
\[
\l_k^{-2} = \inf_{\phi \in H^1} \frac{  \int \rho_0 \left( k^{-2} (\phi')^2 + \phi^2 \right) dy}{\int g\rho_0' \phi^2 dy} 
\]
Let $\phi_k$ be the solution. It is a unique minimizer. Indeed, the corresponding ODE \eqref{e:Rayl} has explicit Wronskian, $W(y) \rho_0(y) = W(0) \rho_0(0)$. Thus, if one solution decays to zero the other one has to grow, thus not in  $H^1$. Let us define the operator
\[
L_c \phi = - (\r_0 \phi')' + k^2 \r_0 \phi  +  \frac{g \r'_0 }{c^2} \phi.
\]
and functional
\[
\begin{split}
F&(\e,c,w) = - (\r_0 (\phi_k + w)')' + k^2 \r_0 (\phi_k + w)  + \e \frac{(\r_0 U')'}{\e U - c}( \phi_k + w) + \frac{g \r'_0 }{( \e U - c)^2}(\phi_k + w)\\
F&: \R \times \C^+ \times H^{m+2} \ra H^m,
\end{split}
\]
where $\C^+ = \{ c \in \C: \im c >0\}$. Then $F$ is a $C^1$ map. 
Note that $F(0,c,w) = L_c(\phi_k + w)$, and hence $F(0,c_k,0) = L_{c_k}\phi_k = 0$. Taking variational derivatives we obtain
\begin{equation}\label{}
\begin{split}
D_c F(0,c_k,0) \a &= - \a \frac{1}{c_k^3} g \rho' \phi_k \\
D_w F(0,c_k,0) w & = L_{c_k} w.
\end{split}
\end{equation} 
To avoid degeneracy let us restrict the functional $F$ to $W = H^{m+2} \ominus [\phi_k]$. Then $D_{c,w} F(0,c_k,0) : W \ra H^m$ is an isomorphism. Indeed, first, the range of $L_{c_k}$ is one-codimensional in $H^m$ since $L_{c_k}$ is self-adjoint and has a one-dimensional kernel. The element $ \rho_0' \phi_k$ does not belong to the range, for if it does we would have had equality $L_{c_k} w = \rho_0' \phi_k$ for some $w \in W$. Testing it  with $\phi_k$ and using symmetry of $L_{c_k}$ we obtain $\int \rho_0' |\phi_k|^2 dy = 0$. However, by the assumption of the Lemma $\rho'_0 >0$ pointwise, implying $\phi_k = 0$. Hence, the range of $D_{c,w} F(0,c_k,0)$ is $H^m$. Injectiveness follows similarly. By the Implicit Function Theorem, there exists $\e_0>0$ such that for all $|\e|<\e_0$ there exists a unique solution to $F(\e,c_\e,w_\e) = 0$. The higher order smoothness of $w_\e$ is trivial from elliptic estimates. The corresponding $r_\e$ can be restored from the second equation in \eqref{e:Rayl} using that $U-c$ never vanishes for $c\in \C^+$.
\end{proof}

With the results of the previous section we see that the theorem provides plenty of examples of steady states with neutral essential spectrum in $\Ldiv$, yet non-trivial unstable discrete spectrum.

%\bibliographystyle{plain}
%\bibliography{rt}

\end{document}